\documentclass[12pt]{article} 

\usepackage{amsmath,amsthm,amsfonts}     
\usepackage{graphicx}
\usepackage{hyperref} 
\usepackage{url}
\usepackage[margin=1in]{geometry} 
\usepackage{tikz}
\usetikzlibrary {arrows.meta}

\newcommand\car[3]{
\draw[rounded corners=3pt,ultra thick,shift={#1},fill=#3] (0,0)--(2,0)--(2,0.5)--(1.5,0.65)--(1.25,1)--(.25,1)--cycle;
\draw[line width=2.5pt,color=black!60!white,fill=white,shift={#1}] (0.5,0) circle (0.2) (1.5,0) circle (0.2);
\node[shift={#1}] at (1,0.5) {\huge $\textbf#2$};
} 

\def\CARPIC{
\begin{tikzpicture}[scale=0.4,transform shape]

\foreach \y in {0,-3.5,...,-17.5}{
    \draw[line width=1.5pt, shift={(0,\y)}] (11.75,-0.25)--(11.75,1.5);
    \draw[line width=1pt, fill=white, shift={(0,\y)}] (11.25,1)--(12,1)--(12,0.9)--(12.35,1.15)--(12,1.4)--(12,1.3)--(11.25,1.3)--cycle;
    \draw[ultra thick, shift={(0,\y)}]
    (23.75 ,-.1)--(23.75,-0.25)--(12.5,-0.25)--(12.5,-.1)  (14.75,-0.25)--(14.75,-.1)(17,-0.25)--(17,-.1) (19.25,-0.25)--(19.25,-.1) (21.5,-0.25)--(21.5,-.1);
    \node[below, shift={(0,\y)}] at (13.625,-0.375) {\huge $\textbf1$};
    \node[below, shift={(0,\y)}] at (15.875,-0.375) {\huge $\textbf2$};
    \node[below, shift={(0,\y)}] at (18.125,-0.375) {\huge $\textbf3$};
    \node[below, shift={(0,\y)}] at (20.375,-0.375) {\huge $\textbf4$};
    \node[below, shift={(0,\y)}] at (22.625,-0.375) {\huge $\textbf5$};
} 

\car{(0,0)}1{magenta!15!white}
\car{(2.25,0)}3{yellow!15!white}
\car{(4.5,0)}2{blue!15!white}
\car{(6.75,0)}4{green!15!white}
\car{(9,0)}2{cyan!15!white}

\car{(2.25,-3.5)}1{magenta!15!white}
\car{(4.5,-3.5)}3{yellow!15!white}
\car{(6.75,-3.5)}2{blue!15!white}
\car{(9,-3.5)}4{green!15!white}
\car{(14.75,-3.5)}2{cyan!15!white}

\car{(4.5,-7)}1{magenta!15!white}
\car{(6.75,-7)}3{yellow!15!white}
\car{(9,-7)}2{blue!15!white}
\car{(19.25,-7)}4{green!15!white}
\car{(14.75,-7)}2{cyan!15!white}

\car{(6.75,-10.5)}1{magenta!15!white}
\car{(9,-10.5)}3{yellow!15!white}
\car{(17,-10.5)}2{blue!15!white}
\car{(19.25,-10.5)}4{green!15!white}
\car{(14.75,-10.5)}2{cyan!15!white}

\car{(9,-14)}1{magenta!15!white}
\car{(21.5,-14)}3{yellow!15!white}
\car{(17,-14)}2{blue!15!white}
\car{(19.25,-14)}4{green!15!white}
\car{(14.75,-14)}2{cyan!15!white}

\car{(12.5,-17.5)}1{magenta!15!white}
\car{(21.5,-17.5)}3{yellow!15!white}
\car{(17,-17.5)}2{blue!15!white}
\car{(19.25,-17.5)}4{green!15!white}
\car{(14.75,-17.5)}2{cyan!15!white}

\end{tikzpicture}
}

\newtheorem{theorem}{Theorem}[section]
\newtheorem{lemma}{Lemma}[section]
\newtheorem{proposition}{Proposition}[section]
\newtheorem{corollary}{Corollary}[section]

\theoremstyle{definition}
\newtheorem{definition}{Definition}[section]
\newtheorem{question}{Question}[section]
\newtheorem{conjecture}{Conjecture}[section]
\newtheorem{observation}{Observation}[section]
\newtheorem{remark}{Remark}[section]

\allowdisplaybreaks

\begin{document}

\title{Lucky cars and lucky spots in parking functions}

\author{Steve Butler\footnote{Department of Mathematics, Iowa State University, Ames, IA 50011 USA\newline \texttt{\{butler,kph3,vlenius,pmart29,mgmoats\}@iastate.edu}} \and Kimberly Hadaway\footnotemark[1] \and Victoria Lenius\footnotemark[1] \and Preston Martens\footnotemark[1] \and Marshall Moats\footnotemark[1]}

\date{\empty}

\maketitle

\begin{abstract}
Parking functions correspond with preferences of $n$ cars which enter sequentially to park on a one-way street where (1) each car parks in the first available spot greater than or equal to its preference and (2) all cars successfully park. When a car parks in its preferred spot then the corresponding car and corresponding spot are deemed ``lucky.'' This paper looks briefly at lucky cars which have previously been studied and in simple cases can be understood by a generalization of a result due to Pollak. We also consider lucky spots where the situation is more complex and not previously studied. Probabilities and asymptotics for lucky spots are given for the first few spots on the one-way street. We close with an exploration of the special cases when cars enter the one-way street in either weakly-increasing or weakly-decreasing order of their preferences.
\end{abstract}

\section{Introduction}
Consider a one-way street with $n$ available parking spots labeled $1,\ldots,n$ (where $1$ is at the start and $n$ is at the end). There are $n$ cars which will sequentially enter to park on this street; each car with its own preference on where to park. As each car enters, it first goes to its preferred spot. If this spot is available, it parks; otherwise, the car moves forwards and parks in the first available spot it finds (if any). If the car reaches the end of the street and fails to park, then it will drive away. An example of this process unfolding is shown in Figure~\ref{fig:parkingexample} where the preferences of the cars (in order) to park are $(2,4,2,3,1)$.

\begin{figure}[htb!]
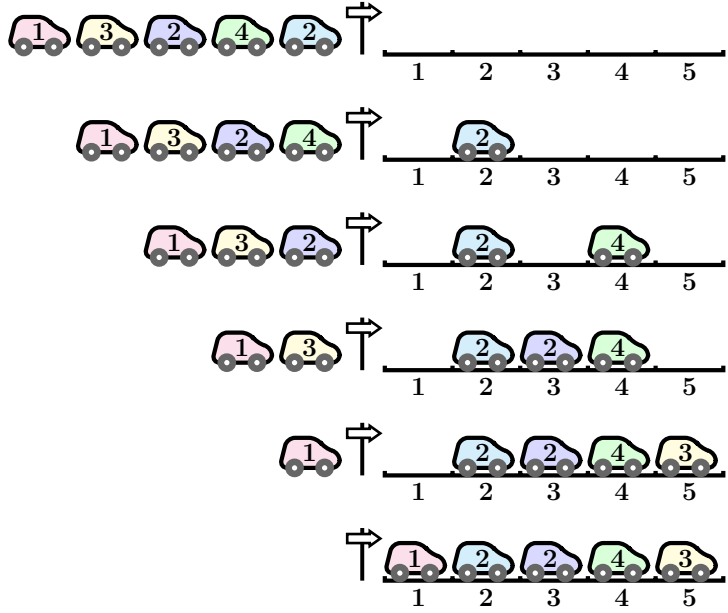

\centering
\CARPIC
\caption{An example of the parking process with preferences $(2,4,2,3,1)$.}
\label{fig:parkingexample}
\end{figure}

\begin{definition}
A \emph{parking function} of $1,\ldots, n$ is a set of preferences for the $n$ cars such that every car is able to park in a one-way street with $n$ labeled spots given that each car parks in the first available spot coming on or after its preference. 
\end{definition}

Parking functions are well studied objects with connections to a variety of topics in enumerative combinatorics (see the survey by Yan \cite{cyan}). Many variations and additional statistics about parking functions have been considered; this paper will focus on lucky cars and lucky spots. The notion of ``lucky'' cars comes from the work of  Gessel and Seo \cite{GS}.)

\begin{definition}
Given a parking function, a car is called a \emph{lucky car} if it is able to park in its preferred spot. Similarly, a spot is called a \emph{lucky spot} if it is parked in by a car which prefers that spot.
\end{definition}

By way of example, for the parking function illustrated in Figure~\ref{fig:parkingexample} the first, second, and fifth cars are lucky. For the same parking function, the first second, and fourth spots are lucky. In general, the indices for lucky cars and lucky spots do not have to agree. However, we do have the following observation.

\begin{observation}
For any particular parking function, the number of lucky cars is equal to the number of lucky spots. In particular, they have the same moments for their probability distribution.
\end{observation}

The paper proceeds as follows. In Section~\ref{sec:Pollak}, we discuss a generalization of Pollak and change it for the setting of lucky cars where we show how to compute moments. In Section~\ref{sec:lucky}, we then give results about lucky cars and lucky spots. In particular, we give the probability the $i$-th car is lucky for all $i$ and that the the $j$-th spot is lucky for $1\le j\le 5$ and $j=n$. Finally, in Section~\ref{sec:weakly}, we consider the problem of lucky cars and spots when the cars enter to park in the one-way street with either weakly-increasing or weakly-decreasing preferences.

\section{A generalization of a result of Pollak}\label{sec:Pollak}
One of the earliest, and most well-known results, on parking functions is that the number of parking functions is $(n+1)^{n-1}$. This result is frequently attributed to  Pollak (communicated in a paper of Riordan \cite{pollak}). The proof readily generalizes to handle additional assumptions on which cars are lucky and which are unlucky (an \emph{unlucky car} is a car which does not park in its preferred space). In the proof below, the case $L=U=\emptyset$ reduces to Pollak's original result and proof.

\begin{theorem}\label{thm:pollak}
Let $L,U\subseteq[n]=\{1,\ldots,n\}$ with $L\cap U=\emptyset$. Then, the number of parking functions where the cars in positions in $L$ park in their preferred spot, the cars in positions in $U$ do not park in their preferred spot, and the remaining cars have no restrictions is
\[
\bigg(\prod_{i\in L}(n+2-i)\bigg)
\bigg(\prod_{i\in U}(i-1)\bigg)(n+1)^{n-|L|-|U|-1}.
\]
\end{theorem}
\begin{proof}
We add an extra spot, $n+1$, and arrange the spots  sequentially in a circle (e.g.\ treating spots as modulo $n+1$). Cars can now prefer any spot on the circle including $n+1$, and instead of exiting, cars continue around the circle until a spot is found. Since there are $n$ cars and $n+1$ spots, all cars can park.

We now consider what happens with the $i$-th car. If  $i \in L$, then $i$ cannot prefer a spot that has already been taken, which gives $n+1-(i-1)=n+2-i$ options to choose from. If $i \in U$, then $i$ must prefer one of the spots that has already been taken, giving $i-1$ options to choose from. Finally, if $i\notin L\cup U$, then there are $n+1$ options to choose from. Importantly, the number of options available does not depend on the choices made by the previous cars. Thus, the number of parking functions on the circle is given by
\[
\bigg(\prod_{i\in L}(n+2-i)\bigg)
\bigg(\prod_{i\in U}(i-1)\bigg)(n+1)^{n-|L|-|U|}.
\]

We can naturally group the parking functions we have created into groups of size $n+1$ where we increase all choices by a fixed amount, modulo $n+1$. This uniform shifting will not change which cars are lucky/unlucky. For each of our groups there will be exactly one representative which avoids $n+1$ (since the parking results are all cyclic shifts). We are interested in the parking functions which avoid $n+1$ (i.e., no car prefers $n+1$ and no ``exit'' occurs from the first $n$ spots). Hence, we can conclude that $1/(n+1)$ fraction of all the parking functions on the circle are our desired parking functions. Multiplying the product above by this factor gives our desired count.
\end{proof}

We can count the number of parking functions associated with a particular assignment of lucky/unlucky to the cars. It is also possible to combine all such assignments together into a polynomial.

\begin{corollary}
Let $c_k$ be the number of parking functions of $1,\ldots,n$ where there are exactly $k$ lucky cars. Then,
\[
f(x)=\frac1{n+1}\prod_{i=1}^n\big((n+2-i)x+(i-1)\big)=\sum_{k=1}^nc_kx^k.
\]
\end{corollary}
\begin{proof}
When multiplying out the $n$ linear terms in the product we will end up with $2^n$ parts, where each part will correspond with a lucky/unlucky assignment. From Theorem~\ref{thm:pollak}, the $i$-th car being lucky contributes a term of $(n+2-i)$ while being unlucky contributes a term of $(i-1)$, i.e.,\ the $x$ terms in the product correspond with the car being lucky and the constant terms with the car being unlucky. Hence, in the full expansion (including the term $1/(n+1)$ in the front), the coefficient of $x^k$ is the total number of parking functions which have $k$ cars which are lucky.
\end{proof}

We note the function $f(x)$ appeared in work of Gessel and Seo \cite{GS} in relation to lucky parking functions (see also \cite{diaconis, prob, NT}).

We can use this function together with the product rule for derivatives from calculus to find expressions for the $c_k$. We have
\begin{equation}\label{eq:ck}
c_k=\frac{f^{(k)}(0)}{k!}=\frac1{n+1}\sum_{\substack{ S\subseteq[n]\\|S|=k}}\bigg(\prod_{i\in S}(n+2-i)\bigg)\bigg(\prod_{i\notin S}(i-1)\bigg),
\end{equation}
where the elements of $S$ correspond to terms we differentiated and the elements not in $S$ are terms where we did not differentiate. The $k!$ will cancel with the $k!$ coming from the product rule as there are $k!$ ways that the order of the derivatives could have occurred. It is straightforward by \eqref{eq:ck}, or directly, to show that $c_1=(n-1)!$ and $c_n=n!$. Using \eqref{eq:ck} it can also be shown that
\[
c_2=(n+1)(n-1)!H_{n-1}-(n-1)(n-1)!
\text{ and }
c_{n-1}=(n+1)!H_n-2n{\cdot}n!,
\]
where $H_n=\frac11+\cdots+\frac1n$ is the $n$-th harmonic number. (We note $c_{n-1}$ are also the second-order Eulerian numbers $\big\langle\!\big\langle{n\atop n-2}\big\rangle\!\big\rangle$; see \texttt{A002538} in the OEIS \cite{oeis}.)

The preceding looks at what happens when we look at the various derivatives of $f(x)$ at $x=0$. We also gain information by looking at these derivatives at $x=1$.

\begin{proposition}
Let $X$ be the random variable for the number of lucky cars for a parking function. Then, $E\big(X(X-1)\cdots(X-(\ell-1))\big)=\frac{1}{(n+1)^{n-1}}f^{(\ell)}(1)$.
\end{proposition}
\begin{proof}
Let $p_\ell=c_\ell/(n+1)^{n-1}$ be the probability that there are $\ell$ lucky cars. Then, taking the derivative and evaluating at $x=1$ we have
\begin{align*}
\frac{1}{(n+1)^{n-1}}f^{(\ell)}(1)&=\sum_{k=1}^nk(k-1)\cdots(k-(\ell-1))p_k\\&=E\big(X(X-1)\cdots(X-(\ell-1))\big).\qedhere
\end{align*}
\end{proof}

With the preceding proposition, we can compute various moments. As an example, we have that the expected number of lucky cars, i.e.\ the mean, is
\begin{align*}
\mu=E(X)&=\frac1{(n+1)^n}\frac{d}{dx}\bigg(\prod_{i=1}^n\big((n+2-i)x+(i-1)\big)\bigg)\bigg|_{x=1}\\
&=\frac1{(n+1)^n}\sum_{i=1}^n(n+2-i)(n+1)^{n-1}=\frac{1}{n+1}\sum_{i=1}^n(n+2-i)\\
&=\frac{1}{n+1}\bigg(n(n+2)-\frac{n(n+1)}2\bigg)=\frac{n(n+3)}{2(n+1)}.
\end{align*}
With a bit more work, one can also show 
\[
\sigma = \sqrt{\frac{(n-1)n(n+4)}{6(n+1)^2}},
\]
where $\sigma^2=E(X^2)-\big(E(X)\big)^2$ is the variance. We leave the details as an exercise for the interested reader (see also \cite{diaconis,monthly}).

\section{Lucky cars and lucky spots}\label{sec:lucky}

First, we introduce a parameter that takes into consideration both spots and cars.

\begin{definition}
Given $1\le i,j\le n$, let $q_n(i,j)$ denote the number of parking functions of $1,\ldots,n$ where the $i$-th car prefers the $j$-th spot \emph{and} the $i$-th car is lucky.
\end{definition}

As an example, the data for $q_7(i,j)$ is provided in Table~\ref{tab:7unordered}. Looking at the entries, they do not appear to always be well-behaved; for instance, the rows are not always monotonic. However, the borders of the data (e.g., the first/last car/spot) do have ``nice'' behavior (insofar as there are formulas to compute them).

\begin{table}[htb]\centering
\[
\begin{array}{|c|@{\quad}c@{\quad}c@{\quad}c@{\quad}c@{\quad}c@{\quad}c@{\quad}c|} \hline
q_7(i,j)&j = 1&j = 2&j = 3&j = 4&j = 5&j = 6&j = 7\\ \hline &&&&&&&\\[-10pt]
i = 1&65536& 48729& 40953& 35328& 30208& 24583& 16807\\ 
i = 2&53248& 41243& 35627& 31502& 27662& 23287& 16807\\ 
i = 3&43008& 32728& 29869& 27406& 24924& 21866& 16807\\ 
i = 4&34496& 24660& 22967& 22788& 21866& 20256& 16807\\ 
i = 5&27440& 17712& 16055& 16608& 18138& 18312& 16807\\ 
i = 6&21609& 12096& 10125& 10240& 11875& 15552& 16807\\ 
i = 7&16807&  7776&  5625&  5120&  5625&  7776& 16807 \\ \hline
\end{array}
\]
\caption{The number of parking functions of $1,\ldots,7$,  where the $i$-th car prefers \emph{and} parks in the $j$-th spot, this is denoted as $q_7(i,j)$. The row sums correspond with when the $i$-th car is lucky; the column sums correspond with when the $j$-th spot is lucky.}
\label{tab:7unordered}
\end{table}

Before producing the formulas for the borders (bottom, top, right, and left), we will find it convenient to give a count for \emph{partial} parking functions (e.g.,\ preferences for cars where all cars park yet not all spots need to be filled). This is again another generalization of Pollak, and the proof follows the same argument as \cite{pollak} (see also Theorem~\ref{thm:pollak}); so we omit the proof here.

\begin{proposition}\label{prop:partial}
Let $s\le t$. The number of preferences for $s$ cars wanting to park in a one-way street with $t$ available spots so that all cars successfully park in an available sport on or after their preference is $(t+1-s)(t+1)^{s-1}$.
\end{proposition}

\begin{lemma}\label{lem:border}
The values $q_n(i,j)$ satisfy the following.
\begin{align*}
q_n(n,j)&=\displaystyle \big({\textstyle{n-1\atop j-1}}\big)j^{j-2}(n-j+1)^{n-j-1}\tag{\emph{Bottom}}\\
q_n(1,j)&=\displaystyle\sum_{k=j}^n\big({\textstyle{n-1\atop k-1}}\big)k^{k-2}(n-k+1)^{n-k-1}\tag{\emph{Top}}\\
q_n(i,n)&=\displaystyle n^{n-2}\tag{\emph{Right}}\\
q_n(i,1)&=\displaystyle(n+1)^{n-i-1}n^{i-2}(2n+1-i)\tag{\emph{Left}}
\end{align*}
\end{lemma}
\begin{proof}
(\emph{Bottom}) The case of $q_n(n,j)$ corresponds to parking functions where the last entry is $j$ and the preceding $n-1$ terms result in having no car parked in $j$. These are formed by taking a parking function on $1,\ldots,j-1$ (which can be done in $j^{j-2}$ ways), a parking function on $j+1,\ldots,n$ (which can be done in $(n-j+1)^{n-j-1}$ ways), and then interleaving them in all possible ways (which can be done in $\big({n-1\atop j-1}\big)$ ways). (This is \texttt{A298594} in the OEIS \cite{oeis}.) We note in passing that these entries show the last car is more likely to be lucky if its preference is either near the start or the end of the street than if its preference is in the middle of the street.

(\emph{Top}) The case of $q_n(1,j)$ corresponds to parking functions that start with $j$. If we skip the first car and use the remainder of the parking function then the result will be an open space in one of the locations from $j$ through $n$, call that open space $k$. Each such case was handled in the previous scenario (i.e.\ union of two parking functions with a gap between), so we sum the previous result from $j$ to $n$. (This is \texttt{A298592} in the OEIS \cite{oeis}.)

(\emph{Right}) The case of $q_n(i,n)$ corresponds to parking functions with $n$ in the $i$-th slot. By the nature of parking functions, there can be at most one $n$ among the preferences, and if there is a car which prefers $n$, then the remaining cars can never move into the $n$-th spot. This means that the remaining $n-1$ slots must form a parking function of $1,\ldots,n-1$ which can be done in $n^{n-2}$ ways.

(\emph{Left}) The case of $q_n(i,1)$ corresponds to parking functions where the \emph{first} occurrence of $1$ is in slot $i$. To build such a parking function, we first place $1$ in slot $i$. We then can place as many additional $1$s in slots $i+1,\ldots,n$ as we want; let $k$ be the number of additional cars with preference $1$ with $0\le k\le n-i$, where for each $k$, there are $\big({n-i\atop k}\big)$ ways to pick which of the cars from $i+1,\ldots,n$  prefer spot $1$. Finally, the remaining $n-k-1$ slots are filled up by cars whose preferences come from $2,\ldots,n$, and when restricted to only those cars, all cars park; by Proposition~\ref{prop:partial}, with $s=n-k-1$ and $t=n-1$ there are $(k+1)n^{n-k-2}$ such preferences. Putting this together and using the binomial theorem (e.g.,\ $\sum_k x^k{m\choose k}=(1+x)^m$ and $\sum_{k}kx^{k-1}{m\choose k}=m(1+x)^{m-1}$), we have
\begin{align*}
\sum_{k=0}^{n-i}\big({\textstyle{n-i\atop k}}\big)(k+1)n^{n-k-2}&=
n^{n-3}\sum_{k=0}^{n-i}k\big(\tfrac1n\big)^{k-1}\big({\textstyle{n-i\atop k}}\big)+
n^{n-2}\sum_{k=0}^{n-i}\big(\tfrac1n\big)^k\big({\textstyle{n-i\atop k}}\big)\\
&=n^{n-3}(n-i)\big(1+\tfrac1n\big)^{n-i-1}+n^{n-2}\big(1+\tfrac1n\big)^{n-i}\\
&=(n+1)^{n-i-1}n^{i-2}\big((n-i)+(n+1)\big).
\qedhere
\end{align*}
\end{proof}
In the last case, we did not need to worry about how the choices for where we put the $1$s interacted with the choice of the preferences of the remaining cars. This is because any rearrangement of parking functions is always a parking function (see \cite{cyan}), and therefore, we can have all of the cars that prefer $1$ park at the end. In this case, we see that they fill in all of the spots which remain unoccupied after the other cars park.

We now consider the row and column sums for the $q_n(i,j)$. By applying Theorem~\ref{thm:pollak}, with $L=\{i\}$ and $U=\emptyset$, we have the following for the $i$-th row sum which corresponds with the $i$-th car being lucky.

\begin{corollary}\label{thm:luckycars}
The number of parking functions of $1,\ldots,n$ for which the $i$-th car is lucky~is
\[
\sum_{j=1}^nq_n(i,j)=(n+2-i)(n+1)^{n-2}=\big(1-\tfrac{i-1}{n+1}\big)(n+1)^{n-1}.
\]
\end{corollary}

Thus, the probability that the $i$-th car is lucky is $1-\frac{i-1}{n+1}$. If we fix a value of $i$ and let $n$ get large, then that probability goes to $1$. In other words, when there are many cars to park, it almost always happens that the early cars get their preferred spots. Similarly, the end cars almost never get their preferred spots.

For the column sums, the situation is starkly different. In Table~\ref{tab:columns}, we have produced the data for the column sums for some small values of $j$ and $n$, and this triangle appears as \texttt{A374756} in the OEIS \cite{oeis}.

\begin{table}[htb]\centering
\[
\begin{array}{c@{\qquad}c@{\quad}c@{\quad}c@{\quad}c@{\quad}c@{\quad}c} \hline
n&j = 1&j = 2&j = 3&j = 4&j = 5 &j=6\\
\hline \\[-10pt]
1&1\\ 
2&3& 2\\
3&16& 11& 9\\ 
4&125& 87& 74& 64\\ 
5&1296& 908& 783& 708& 625\\ 
6&16807& 11824& 10266& 9421& 8733 &7776 \\
7&262144& 184944& 161221& 148992& 140298 & 131632\\ 
8&4782969& 3381341& 2955366& 2742090& 2600879 & 2480787 \\ 
9&100000000& 70805696& 61999923&  57671104& 54921875 & 52779840\\ 
10 & 2357947691 & 1671605646 & 1465709426 & 1365730231 & 1303885965 & 1258181726 \\
\hline
\end{array}
\]
\caption{The number of parking functions for $n$ cars where the $j$-th spot is lucky for $1\le j\le 6$ and for $1\le n\le 10$.}
\label{tab:columns}
\end{table}

Since there are no spots before $1$, then the only car that can park in spot $1$ must prefer spot $1$, meaning that spot $1$ is always lucky. We have the following.
\begin{observation}
The number of parking functions of $1,\ldots,n$ for which the first spot is lucky is
\[
\sum_{i=1}^nq_n(i,1)=(n+1)^{n-1}.
\]
\end{observation}

We now turn to the sum of the entries in the second column, in other words, the number of times the second spot is lucky. 

\begin{theorem}\label{thm:spot2}
The number of parking functions of $1,\ldots,n$ for which the second spot is lucky is
\[
\sum_{i=1}^nq_n(i,2)=\tfrac34(n+1)^{n-1}-\tfrac14(n-1)^{n-1}.
\]
\end{theorem}
\begin{proof}
Let $\ell$ denote the number of cars that prefer $1$ or $2$ (so $2\le \ell\le n$), and let $k$ denote the number of cars that prefer $1$ (so $1\le k\le \ell-1$). For each choice of $k$ and $\ell$, we do the following:
\begin{itemize}
\item Choose which of the cars will prefer either $1$ or $2$, this can be done in $\big({n\atop\ell}\big)$ ways.
\item For each choice that we have just made, now select how the $1$s and the $2$s will be relatively placed. The key observation is that the first $2$ must happen before two $1$s have occurred. (If not, the second spot cannot be lucky.) We can do this by looking at all ways to position the $1$s (which is $\big({\ell\atop k}\big)$) and then removing the ways which have two $1$s at the start (which is $\big({\ell-2\atop k-2}\big)$). Altogether, this can be done in $\big({\ell\atop k}\big)-\big({\ell-2\atop k-2}\big)$ ways.
\item Assign the remaining $n-\ell$ cars to park in positions $3,\ldots,n$ so that among them there are no exits. Applying Proposition~\ref{prop:partial} with $s=n-\ell$ and $t=n-2$, we have that this can be done in $(\ell-1)(n-1)^{n-\ell-1}$ ways.
\end{itemize}

Putting everything together we have
\begin{align*}
\sum_{i=1}^nq_n(i,2)&=\sum_{\ell=2}^n\sum_{k=1}^{\ell-1}
 {n\choose \ell}\bigg({\ell\choose k}-{\ell-2\choose k-2}\bigg)(\ell-1)(n-1)^{n-\ell-1}\\
 &=(n-1)^{n-1}\sum_{\ell=2}^n(n-1)^{-\ell}(\ell-1){n\choose \ell}\sum_{k=1}^{\ell-1}
 \bigg({\ell\choose k}-{\ell-2\choose k-2}\bigg).
\end{align*}
We start with the inner sum:
\begin{align*}
\sum_{k=1}^{\ell-1}
 \bigg({\ell\choose k}-{\ell-2\choose k-2}\bigg)&=
 \bigg(\sum_{k=0}^{\ell}
 {\ell\choose k}-1-1\bigg)-\bigg(\sum_{k=2}^{\ell}{\ell-2\choose k-2}-1\bigg)\\
 &=(2^\ell-2)-(2^{\ell-2}-1)=\tfrac342^\ell-1.
\end{align*}
Replacing the inner sum and then expanding and separating, the outer sum  becomes the following (note that at $\ell=1$ the expression in the sum is $0$ and at $\ell=0$ the expression in the sum is $\frac14$):
\begin{multline*}
\sum_{\ell=2}^n(n-1)^{-\ell}(\ell-1){n\choose \ell}\big(\tfrac34\cdot2^\ell-1\big)=
\sum_{\ell=0}^n(n-1)^{-\ell}(\ell-1){n\choose \ell}\big(\tfrac34\cdot2^\ell-1\big)-\tfrac14\\
=\tfrac34\sum_{\ell=0}^n\ell\big(\tfrac2{n-1}\big)^\ell{n\choose \ell}
-\tfrac34\sum_{\ell=0}^n\big(\tfrac2{n-1}\big)^\ell{n\choose \ell}
-\sum_{\ell=0}^n\ell\big(\tfrac1{n-1}\big)^\ell{n\choose \ell}
+\sum_{\ell=0}^n\big(\tfrac1{n-1}\big)^\ell{n\choose \ell}-\tfrac14
\end{multline*}
These sums can be evaluated by using the binomial theorem and a derivative variants (e.g.,\ $\sum_\ell x^\ell{m\choose \ell}=(1+x)^m$ and $\sum_{\ell}\ell x^\ell{m\choose \ell}=mx(1+x)^{m-1}$). The outer sum now becomes the following:
\begin{multline*}
\tfrac34n\big(\tfrac2{n-1}\big)\big(\tfrac{n+1}{n-1}\big)^{n-1}
-\tfrac34\big(\tfrac{n+1}{n-1}\big)^{n}
-n\big(\tfrac{1}{n-1}\big)\big(\tfrac{n}{n-1}\big)^{n-1}
+\big(\tfrac{n}{n-1}\big)^{n}
-\tfrac14\\
=\frac{6n(n+1)^{n-1}-3(n+1)^n-(n-1)^n}{4(n-1)^n}
=\frac{3(n+1)^{n-1}-(n-1)^{n-1}}{4(n-1)^{n-1}}.
\end{multline*}
Finally, using this expression for the outer sum, we have
\begin{align*}
\sum_{i=1}^nq_n(i,2)
 &=(n-1)^{n-1}\bigg(\frac{3(n+1)^{n-1}-(n-1)^{n-1}}{4(n-1)^{n-1}}\bigg)=\tfrac34(n+1)^{n-1}-\tfrac14(n-1)^{n-1}.\qedhere
\end{align*}
\end{proof}

Before proceeding, we observe that the resulting expression is surprisingly simple. It would be interesting to see if there is a more elegant way to arrive at the same result.

Similar expressions can be found for other column sums, but the case analysis gets more substantial as $j$ increases. As an example, for the third column sum when we restrict to what the $1$s and $2$s are doing, if it starts as any of $12$, $21$, or $11$, then the first $3$ must occur by the (relative) third term when restricted to $1$s, $2$s, and $3$s. However, if the restriction starts as $22$, then the first $3$ must occur by the (relative) second term when restricted to $1$s, $2$s, and $3$s (i.e.,\ the first car preferring $2$ parks in spot $2$, and the second car preferring $2$ will park in spot $3$ so that the third spot is not parked in by a car which prefers $3$). 

Carrying out the case analysis and simplifying the resulting sums (e.g.,\ by similar manipulation as above), we have the following expressions for the next few column sums (equivalently, the number of parking functions for which the third, fourth, and fifth spots are lucky).
\begin{align*}
\sum_{i=1}^nq_n(i,3)&=\tfrac23(n+1)^{n-1}-\tfrac13(2n-1)(n-2)^{n-2} \\
\sum_{i=1}^nq_n(i,4)&=\tfrac58(n+1)^{n-1}-\tfrac18(13n^2-26n+9)(n-3)^{n-3} \\
\sum_{i=1}^nq_n(i,5)&=\tfrac35(n+1)^{n-1}-\tfrac1{30}(118n^3-531n^2+659n-192)(n-4)^{n-4} 
\end{align*}

Using these expressions, we can also find the asymptotics for how likely it is that the $j$-th spot is lucky for $1\le j\le 5$. This is given in Table~\ref{tab:asymptotics}.

\begin{table}[htb]\centering
\begin{tabular}{c@{\qquad}l@{\qquad}c}\hline &&\\[-10pt]
$j$&Exact $\rho_j$&Numerical $\rho_j$\\ \hline &&\\[-10pt]
$1$&$1$&\texttt{1.000000\ldots}\\[3pt]
$2$&$\frac34-\frac14e^{-2}$&\texttt{0.716166\ldots}\\[3pt]
$3$&$\frac23-\frac23e^{-3}$&\texttt{0.633475\ldots}\\[3pt]
$4$&$\frac58-\frac{13}8e^{-4}$&\texttt{0.595237\ldots}\\[3pt]
$5$&$\frac35-\frac{59}{15}e^{-5}$&\texttt{0.573497\ldots}\\[3pt]
\hline
\end{tabular}
\medskip

\caption{The asymptotic probability as $n\to\infty$ that the $j$-th spot is lucky, this is denoted $\displaystyle\rho_j=\lim_{n\to\infty}\bigg((n+1)^{-(n-1)}\sum_{i=1}^nq_n(i,j)\bigg)$.}
\label{tab:asymptotics}
\end{table}

A pattern seems to be emerging, which suggests the following conjecture.

\begin{conjecture}
The number of parking functions where the $j$-th spot is lucky has the form $\tfrac{j+1}{2j}(n+1)^{n-1}-f_j(n)(n-j+1)^{n-j+1}$ where $f_j(n)$ is a polynomial of degree $j-2$ with rational coefficients. In particular, for some rational $r_j$, the asymptotic probability that the $j$-th spot is lucky is $\frac{j+1}{2j}-r_je^{-j}$.
\end{conjecture}

Looking back at Table~\ref{tab:columns}, one might also notice that the diagonal terms (i.e.,\ the $n$-th column sums) are well-behaved. Indeed, by Lemma~\ref{lem:border}, we have the following:

\begin{observation}
The number of parking functions of $1,\ldots,n$ for which the $n$-th spot is lucky is
\[
\sum_{i=1}^nq_n(i,n)=\sum_{i=1}^nn^{n-2}=n^{n-1}.
\]
\end{observation}

Thus, the probability that the \emph{last} spot is occupied by a car which prefers the last spot is asymptotically $e^{-1}=\texttt{0.367879\ldots}$. By comparison, the probability that the last car parks in its preferred spot is  $\frac2{n+1}\to 0$.

\begin{question}
What can be said about the $(n-1)$-st column sum, $\sum_iq_n(i,n-1)$? Or more generally, the $(n-j)$-th column sum, $\sum_iq_n(i,n-j)$? 
\end{question}
For reference the $(n-1)$-st column sums (the sub-diagonal in Table~\ref{tab:columns}) starts with the following values:
\[
3,~ 11,~ 74,~ 708,~ 8733,~ 131632,~ 2342820,~ 48068672,~ 1116809255,~ \ldots.
\]
The authors were not able to find expressions for these sums. Moreover, this sequence (along with columns $\ge 2$ in Table~\ref{tab:columns}) has not previously appeared in the OEIS \cite{oeis}.
Thus, they have added these sequences to OEIS \cite{oeis}: \texttt{A372842}, \texttt{A372843}, \texttt{A372844}, \texttt{A372845}, \texttt{A374533}, and \texttt{A374756}. 
Moreover, the row sums indicate the number of lucky cars across all parking functions of length $n$, and this yields another new sequence (now \texttt{A375616} in OEIS \cite{oeis}) whose formula is in terms of other sequences in OEIS.
There is still much to be found with regards to lucky spots.

\section{Weakly-increasing and weakly-decreasing}\label{sec:weakly}
We now look at the special cases when the preference of cars are either in weakly-increasing or weakly-decreasing order.  It is known that the number of such parking functions in these cases is the $n$-th Catalan number $C_n=\frac1{n+1}{2n\choose n}$ (see \cite{stanley, cyan}). This can be seen by using the classical bijection between such parking functions and Dyck paths which we describe below.

\begin{definition}
A \emph{Dyck path} of size $n$ is a path going from $(0,0)$ to $(n,n)$ consisting of unit north and unit east steps never going below the line $y=x$.
\end{definition}

For the bijection between Dyck paths and weakly-increasing parking functions, start with a Dyck path, and label the \emph{columns (spots)} from left to right with $1,\ldots,n$. Similarly, label the \emph{rows (cars)} from bottom to top with $1,\ldots,n$. To find the preference of the $i$-th car, go from the right hand side of the row labeled $i$ moving to the left until you hit the Dyck path and go down to read the column label. For example, in Figure~\ref{fig:incdec}(a) we go from a Dyck path of size $n=8$ to the weakly-increasing parking function $(1,1,2,2,2,6,7,7)$.

For the bijection of weakly-decreasing parking functions, the only difference is that the rows (cars) are labeled from top to bottom with $1,\ldots,n$. For example, in Figure~\ref{fig:incdec}(b) we go from a Dyck path of size $n=8$ to the weakly-decreasing parking function $(7,7,6,2,2,2,1,1)$.

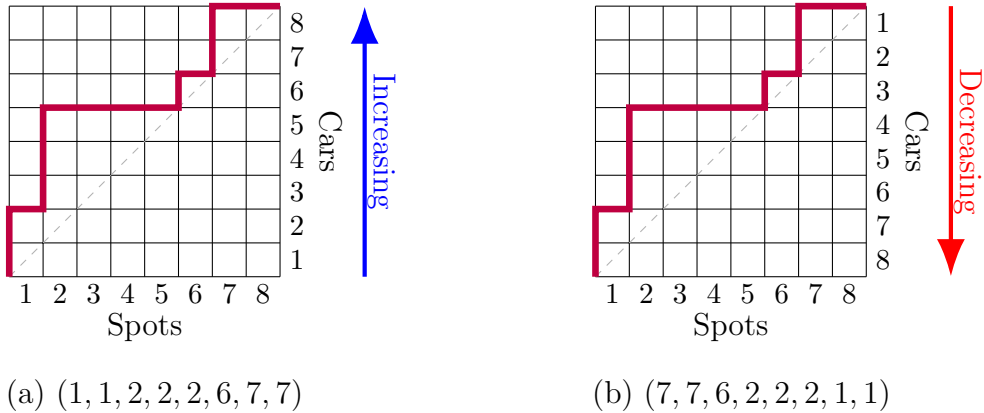
\begin{figure}[htb]\centering
\begin{tabular}{l@{\qquad\qquad\qquad}l}
\begin{tikzpicture}[scale=0.45]
\foreach \x in {0,...,8}
    \draw (\x, 0)--(\x,8) (0,\x)--(8,\x);

\draw[dashed,color=white!70!black] (0, 0)--(8, 8);

\node[] (x-axis-label) at (4, -1.5) {Spots};
\node[rotate=-90] (y-axis-label) at (9.5, 4) {Cars};

\foreach \x in {1,...,8} {
    \node at ({\x - .5}, -.5) {\small$\x$};
    \node at (8.5, {\x-0.5}) {\small$\x$};
    }

\draw [color=blue,ultra thick,-{Latex[length=5mm]}] (10.5,0)--(10.5,8);

\node[rotate=-90,color=blue] (arrow-label) at (11, 4) {Increasing};

        \draw[line width=2.5pt,color=purple] (0, 0)--(0, 2)--(1, 2)--(1, 5)--(5, 5)--(5, 6)--(6, 6)--(6, 8)--(8, 8);
    \end{tikzpicture}
    &
\begin{tikzpicture}[scale=0.45]
\foreach \x in {0,...,8}
    \draw (\x, 0)--(\x,8) (0,\x)--(8,\x);

\draw[dashed,color=white!70!black] (0, 0)--(8, 8);

\node[] (x-axis-label) at (4, -1.5) {Spots};
\node[rotate=-90] (y-axis-label) at (9.5, 4) {Cars};

\foreach \x in {1,...,8} {
    \node at ({\x - .5}, -.5) {\small$\x$};
    \node at (8.5, {8.5-\x}) {\small$\x$};
    }

\draw [color=red,ultra thick,-{Latex[length=5mm]}] (10.5,8)--(10.5,0);

\node[rotate=-90,color=red] (arrow-label) at (11, 4) {Decreasing};

\draw[line width=2.5pt,color=purple] (0, 0)--(0, 2)--(1, 2)--(1, 5)--(5, 5)--(5, 6)--(6, 6)--(6, 8)--(8, 8);
\end{tikzpicture}

\\[5pt]
(a) $(1,1,2,2,2,6,7,7)$&(b) $(7,7,6,2,2,2,1,1)$
\end{tabular}
\caption{Bijection between a Dyck path and a (a) weakly-increasing parking function, and a (b) weakly-decreasing parking function.}
\label{fig:incdec}
\end{figure}

For this discussion, we modify the notation for the previously introduced parameters.

\begin{definition}
Given $1\le i,j\le n$, let $q_n^{\textrm{i}}(i,j)$ (similarly, $q_n^{\textrm{d}}(i,j)$) denote the number of increasing (similarly, decreasing) parking functions of $1,\ldots,n$ where the $i$-th car prefers the $j$-th spot and the $i$-th car is lucky.
\end{definition}

\subsection{Weakly-increasing}
When the cars come in weakly-increasing order, the spots in the street fill up in order from $1$ through $n$. This means that the $i$-th car to enter parks in spot $i$. In particular, the $i$-th car is lucky if and only if it prefers spot $i$. That is, for weakly-increasing parking functions, the $i$-th car is lucky if and only if the $i$-th spot is lucky.

In terms of Dyck paths, this corresponds to when the Dyck path comes up off of the line $y=x$ on the left-hand side of the column labelled $i$. In particular, the Dyck path must be a concatenation of two Dyck paths: one for the first $i-1$ columns (which can be done in $C_{i-1}$ ways) and one for the remaining $n-i+1$ columns (which can be done in $C_{n-i+1}$ ways). This establishes the following.

\begin{lemma}
We have that
\[
q_n^{\text{\emph{i}}}(i,j)=\begin{cases} 
C_{i-1}C_{n-i+1}&\text{if }i=j,\\
0&\text{else.}
\end{cases}
\]
In particular, the number of weakly-increasing parking functions of $1,\ldots,n$ for which the $i$-th car (spot) is lucky is $C_{i-1}C_{n-i+1}$.
\end{lemma}

\begin{corollary}
The expected number of lucky cars/spots for a random weakly-increasing parking function is $\frac{3n}{n+2}$.
\end{corollary}
\begin{proof}
The expected number of lucky cars/spots is equivalent to the expected number of times a Dyck path comes off of $y=x$, or equivalently, the expected number of times the Dyck path returns to $y=x$. This has previously been computed (see Deutsch \cite{deutsch}) to have expected value $\frac{3n}{n+2}$.
\end{proof}

We note that having the preferences in weakly-increasing order attains the minimum number of lucky cars/spots for a given set of preferences. This can be seen because if the $i$-th car is lucky for a weakly-increasing parking function, then the first $i-1$ cars would have wanted to park in the first $i-1$ spots, while all remaining cars want to park in $i$ or above. In particular, no car which has a preference before \emph{or} after spot $i$ can park in spot $i$. So if a spot is lucky in the weakly-increasing order, it will be lucky in any ordering. In other words, for weakly-increasing parking functions, the only spots that are lucky are those which must always be lucky.

\subsection{Weakly-decreasing}
When the cars come in weakly-decreasing order, then a car is lucky if it parks in a previously unpreferred spot. This can be seen by noting that if a car prefers a previously unpreferred spot, then all preceding cars must have parked further down the street and thus the desired spot is available. In particular, every spot which has at least one car preferring that spot is lucky. This means that the weakly-decreasing order attains the maximum number of lucky cars/spots.

In terms of Dyck paths, the $i$-th car will prefer the $j$-th and be lucky if and only if the Dyck path contains the sub-path shown in Figure~\ref{fig:dlucky}. In terms of coordinates of the lattice, this corresponds to the following pair of edges being present in the path:
\begin{equation}\label{eq:dyck}
(j-1,n-i)\to(j-1,n-i+1)\to(j,n-i+1).
\end{equation}
In particular, the number of lucky cars/spots corresponds with the number of peaks in the Dyck path.

\begin{figure}[htb]
\centering
\begin{tikzpicture}[scale=0.45]
\foreach \x in {0,...,6}
    \draw (\x, 0)--(\x,6) (0,\x)--(6,\x);

\draw[dashed,color=white!70!black] (0, 0)--(6, 6);

\node at (2.5,-.5){\small$j$};
\node at (6.5,4.5){\small$i$};

\draw[line width=2.5pt,color=purple] (2, 4)--(2, 5)--(3, 5);

\draw[line width=3.5pt, color=blue!50!white,opacity=0.25] (2.5,0)--(2.5,4.5)--(6,4.5);
\end{tikzpicture}

\caption{Sub-path needed in the Dyck path for a weakly-decreasing parking function having the $i$-th car prefer the $j$-th spot and be lucky.}
\label{fig:dlucky}
\end{figure}
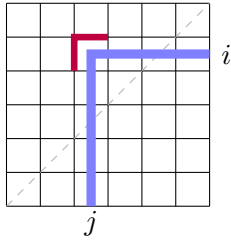

\begin{remark}
eThe Narayana numbers $N(n,k)=\frac1k{n-1 \choose k-1}{n \choose k-1}$ (see \texttt{A001263} in the OEIS \cite{oeis}) count the number of Dyck paths from $(0,0)$ to $(n,n)$ with exactly $k$ peaks. Equivalently, the Narayana numbers are the number of weakly-decreasing parking functions of $1,\ldots,n$ with exactly $k$ lucky cars/spots.
\end{remark}

We now work to find $q_n^{\text{d}}(i,j)$, and we start by counting particular lattice paths.

\begin{lemma}\label{lem:w_dec_lucky_car_count}
For $\ell\ge k$, the number of lattice paths from $(0,0)$ to $(k,\ell)$ which stay on or above the line $y=x$ is $\frac{\ell - k + 1}{\ell + 1}{k + \ell \choose k} $.
\end{lemma}
\begin{proof}
We follow a classic sub-path flipping argument for Dyck paths. We start by noting the total number of paths from $(0,0)$ to $(k,\ell)$ is ${k+\ell \choose k}$, e.g.,\ there are $k$ east steps and $\ell$ north steps, and thus, we only need to decide which of the $k+\ell$ steps are east steps. 

We now need to remove any paths that go below the line $y=x$. To count these, we note that there is a \emph{first} such time that the path hits the line $y=x-1$; then, we flip the remainder of the path across the line $y=x-1$ to form a path that goes from $(0,0)$ to $(\ell+1,k-1)$. Note that this can also be reversed for any path from $(0,0)$ to $(\ell+1,k-1)$. In particular, there are ${(\ell+1)+(k-1)\choose k-1}={k+\ell\choose k-1}$ such paths.

Thus, the number of paths is
\[
{k+\ell \choose k}-{k+\ell\choose k-1}=
\bigg(1-\frac{k}{\ell+1}\bigg){k+\ell \choose k}=
\frac{\ell - k + 1}{\ell + 1}{k + \ell \choose k} .\qedhere
\]
\end{proof}

\begin{lemma}\label{lem:car_i_lucky}
The number of weakly-decreasing parking functions where the $i$-th car prefers the $j$-th spot and is lucky is
\begin{equation}\label{eq:qd}
q_n^{\text{\emph{d}}}(i,j)=
\begin{cases}
\displaystyle\frac{(n - i - j + 2)^2}{(n - i + 1)(n - j + 1)}{n - i + j - 1 \choose j - 1} {n - j + i - 1 \choose i - 1} &\text{if }i+j\le n+1,\\[10pt]
0&\text{else.}
\end{cases}
\end{equation}
\end{lemma}
\begin{proof}
The condition that $i+j\le n+1$ is needed to ensure the peak occurs above the line $y=x$.

We can break our Dyck paths which contribute to $q_n^{\text{d}}(i,j)$ into three parts, namely from $(0,0)$ up to $(j-1,n-i)$, then the two steps in \eqref{eq:dyck}, and finally from $(j,n-i+1)$ to $(n,n)$. By symmetry (e.g.,\ flipping across the ``anti-diagonal''), this last portion is equivalent to counting the walks that stay above $y=x$ from $(0,0)$ to $(i-1,n-j)$. Now, apply Lemma~\ref{lem:w_dec_lucky_car_count} to find the counts for the first and last portions of the Dyck path; multiplying then gives the result.
\end{proof}

\begin{table}[htb]\centering
\[
\begin{array}{|c|@{\quad}c@{\quad}c@{\quad}c@{\quad}c@{\quad}c@{\quad}c@{\quad}c|} \hline
q_7^{\text{d}}(i,j)&j = 1&j = 2&j = 3&j = 4&j = 5&j = 6&j = 7\\ \hline &&&&&&&\\[-10pt]
i=1&  1&   6&  20&  48&  90& 132& 132\\
i=2&  6&  25&  56&  84&  84&  42&   0\\
i=3& 20&  56&  81&  70&  28&   0&   0\\
i=4& 48&  84&  70&  25&   0&   0&   0\\
i=5& 90&  84&  28&   0&   0&   0&   0\\
i=6&132&  42&   0&   0&   0&   0&   0\\
i=7&132&   0&   0&   0&   0&   0&   0\\ \hline
\end{array}
\]
\caption{The number of weakly-decreasing parking functions of $1,\ldots,7$,  where the $i$-th car prefers \emph{and} parks in the $j$-th spot, this is denoted as $q_7^{\text{d}}(i,j)$. The row sums correspond with when the $i$-th car is lucky; the column sums correspond with when the $j$-th spot is lucky.}
\label{tab:7decreasing}
\end{table}

Examining the data in Table~\ref{tab:7decreasing} for $q_7^{\text{d}}(i,j)$, we observe symmetry in the entries. This can be seen directly by examining \eqref{eq:qd} which is symmetric in its use of $i$ and $j$. This can also be seen visually by flipping any Dyck path across the ``anti-diagonal'' and seeing where the peaks go. As a consequence of the symmetry of $q_n^{\text{d}}(i,j)$, we have that the number of times the $i$-th car is lucky is equal to the number of times that the $i$-th spot is lucky.

\begin{corollary}
The number of weakly-decreasing parking functions of length $n$ where the $j$-th spot (car) is lucky is
\begin{equation}\label{eq:sumq}
\sum_{i=1}^nq_n^{\text{\emph{d}}}(i,j)=\sum_{i=1}^{n+1-j}\frac{(n - i - j + 2)^2}{(n - i + 1)(n - j + 1)}{n - i + j - 1 \choose j - 1} {n - j + i - 1 \choose i - 1}.
\end{equation}
\end{corollary}

While \eqref{eq:sumq} gives an expression for the number of times the $j$-th spot (car) is lucky, it is somewhat unwieldy. By approaching the problem of counting from a different direction, we find a more convenient expression (see \eqref{eq:sumq2} below). These expressions are thus equal; the authors do not know a direct algebraic proof that establishes their equality.

\begin{theorem}\label{thm:weak}
The number of weakly-decreasing parking functions of length $n$ where the $j$-th spot (car) is lucky is
\begin{equation}\label{eq:sumq2}
\sum_{k=0}^{n-j}C_{n-1-k}C_k=\sum_{k=0}^{n-j}\frac1{(n-k)(k+1)}{2(n-1-k)\choose n-1-k}{2k\choose k}.
\end{equation}
\end{theorem}

Before beginning the proof, we note that the numbers which arise in this manner correspond with the entries of \texttt{A067323} in the OEIS \cite{oeis} which come from the Catalan triangle.

\begin{proof}
We establish a bijection between $P_n^{(j)}$, a Dyck path of length $n$ where there is a peak on the $j$-th column, and $(Q_{n-1-k}, Q_k)$ an (ordered) pair of Dyck paths of length $n-1-k$ and $k$, respectively, with $0\le k\le n-j$.

First, note that for $P_n^{(j)}$ that there must be at least one north step on the line $x=j-1$. Now let $(j-1,\ell)$ be the \emph{first} point where the $P_n^{(j)}$ hits $x=j-1$ and let $(s,t)$ be the \emph{next} point where $P_n^{(j)}$ hits the line $y=x+\ell-j+1$. 

Let $k=s-j$. The Dyck path $Q_{n-1-k}$ is formed by concatenating the part of $P_n^{(j)}$ from $(0,0)$ to $(j-1,\ell)$ with the part of $P_n^{(j)}$ from $(s,t)$ to $(n,n)$. The Dyck path $Q_k$ is formed by the part of $P_n^{(j)}$ from $(j-1,\ell+1)$ to $(s-1,t)$. An example of this is illustrated in Figure~\ref{fig:bijection}. 

This process can be reversed by taking the the Dyck path $Q_{n-1-k}$ and splitting at $x=j-1$. We have the first part start at $(0,0)$, and the second part end at $(n,n)$. We now add a north step out of the first part, and an east step into the second part. Finally, insert the Dyck path $Q_k$ to combine these two paths. In order to carry this out  it must be that the first Dyck path has to touch the line $x=j-1$ which forces $0\le k\le n-j$.

\begin{figure}[htb]
\centering

\begin{tikzpicture}[scale=0.45]
\fill[color=blue!10!white] (4,6)--(4,9)--(7,9)--cycle;
\fill[color=red!10!white] (8,8)--(8,10)--(10,10)--cycle;
\fill[color=red!10!white] (0,0)--(0,6)--(4,6)--(4,4)--cycle;

\foreach \x in {0,...,10}
    \draw[color=white!60!black] (\x, 0)--(\x,10) (0,\x)--(10,\x);

\draw[dashed,color=white!75!black] (0, 0)--(10, 10);
\draw[dashed,color=white!90!black] (4,6)--(7,9);

\node at (4.5,-.5) {\small$j$};

\draw[line width=2.5pt,color=red] 
(0, 0)--(0, 1)--(1, 1)--(1, 2)--(2, 2)--(2, 5)--(4, 5) (10,10)--(9,10)--(9,9)--(8,9);
\draw[line width=2.5pt,color=green!70!black,densely dotted] 
(4, 5)--(4,6) (8,9)--(7,9);
\draw[line width=2.5pt,color=blue,densely dashed] 
(4,6)--(4,8)--(5,8)--(5,9)--(7,9);

\node[scale=1.5] at (12,5) {$\leftrightarrow$};
\node[yscale=2.75] at (14,5) {$\bigg($};

\fill[color=red!10!white] (15,2)--(15,8)--(21,8)--cycle;

\draw[dashed,color=white!75!black] (15,2)--(21,8);
\foreach \x in {0,...,6}
    \draw[color=white!60!black] (\x+15, 2)--(\x+15,8) (15,\x+2)--(21,\x+2);
\node at (19.5,1.5) {\small$j$};
\draw[line width=2.5pt,color=red] 
(15,2)--(15,3)--(16,3)--(16,4)--(17,4)--(17,7)--(20,7)--(20,8)--(21,8);

\node[scale=2] at (22,3) {$,$};

\fill[color=blue!10!white] (23,2)--(23,5)--(26,5)--cycle;
\draw[dashed,color=white!75!black] (23,2)--(26,5);
\foreach \x in {0,...,3}
    \draw[color=white!60!black] (\x+23, 2)--(\x+23,5) (23,\x+2)--(26,\x+2);
\draw[line width=2.5pt,color=blue,densely dashed] 
(23,2)--(23,4)--(24,4)--(24,5)--(26,5);
\node[yscale=2.75] at (27,5) {$\bigg)$};

\end{tikzpicture}
\caption{Example of a bijection between $P_n^{(j)}\leftrightarrow (Q_{n-1-k},Q_k)$ for Theorem~\ref{thm:weak}.}
\label{fig:bijection}
\end{figure}
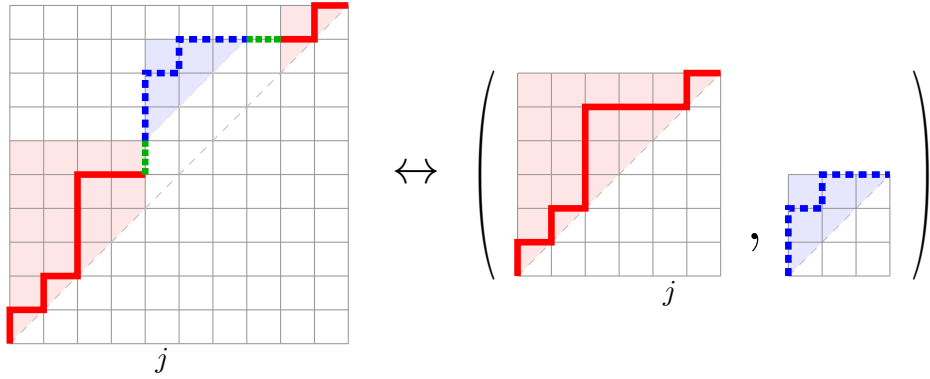

Since \eqref{eq:sumq2} counts the number of ways to select the two Dyck paths $(Q_{n-1-k},Q_k)$ for $0\le k\le n-j$, the result follows.
\end{proof}

Finally, let us compute the expected number of lucky spots for weakly-decreasing parking functions. As noted earlier, this is the same as the number of preferred spots for a weakly-decreasing parking function.

\begin{theorem}\label{thm:expect_coin_flip}
For $n\ge 1$, the expected number of lucky spots for a weakly-decreasing parking function is $\frac12(n+1)$.
\end{theorem}
\begin{proof}
We use Theorem~\ref{thm:weak} to count the total number of lucky spots over all weakly-decreasing parking functions of length $n$ to get
\[
\sum_{j=1}^{n} \bigg(\sum_{k=0}^{n-j} C_{n-1-k}C_{k}\bigg)
=\sum_{k=0}^{n-1} \bigg(\sum_{j=1}^{n-k} C_{n-1-k}C_{k}\bigg)
=\sum_{k=0}^{n-1}(n-k)C_{n-1-k}C_{k}.
\]

We recall (see \cite{stanley}) that the generating function for Catalan numbers is 
\[
f(x)=\sum_{n= 0}^\infty C_nx^n=\frac{1-\sqrt{1-4x}}{2x} 
\text{ ~~and so~~ }
\frac{d}{dx}\big(xf(x)\big)=\sum_{n=0}^{\infty} nC_{n-1}x^n=\frac{x}{\sqrt{1-4x}}.
\]
Putting this together, we have
\begin{multline*}
\sum_{n=0}^\infty\bigg(\sum_{k=0}^{n-1}(n-k)C_{n-1-k}C_{k}\bigg)x^n
= \bigg( \sum_{n=0}^{\infty} C_{n} x^n \bigg) \bigg( \sum_{n=0}^{\infty} n C_{n-1} x^n \bigg)\\=\bigg(\frac{1-\sqrt{1-4x}}{2x}\bigg)\bigg(\frac{x}{\sqrt{1-4x}}\bigg)=\frac{1}{2\sqrt{1 - 4x}} - \frac{1}{2}=\sum_{n =1}^{\infty} \frac{1}{2}{2n\choose n} x^n.
\end{multline*}
In particular, the total number of lucky spots is $\frac12{2n\choose n}$. On the other hand, the total number of weakly-decreasing parking functions is $\frac{1}{n+1}{2n\choose n}$. Dividing the two terms gives the result.
\end{proof}

\end{document}